\documentclass[amssymb,amsfonts,reqno]{amsart}
\usepackage{amssymb,latexsym}
\usepackage{xypic}
\usepackage{verbatim}
\usepackage[colorlinks,
linkcolor=red,
anchorcolor=blue,
citecolor=blue]{hyperref}

\usepackage[numbers,sort&compress]{natbib}
\allowdisplaybreaks \allowdisplaybreaks[4]
\include{psfig}

\begin{document}

\title{On weighted compactness of commutators of Schr\"{o}dinger operators}

\author{Qianjun He}
\address{School of Applied Science, Beijing Information Science and Technology University, Beijing, 100192, P. R. China.}
\email{qjhe@@bistu.edu.cn}

\author{Pengtao Li}
\address{College of Mathematics, Qingdao University, Qingdao, Shandong, 266071, P. R. China.}
\email{ptli@@qdu.edu.cn}


\begin{abstract}
Let $\mathcal{L}=-\Delta+\mathit{V}(x)$ be a Schr\"{o}dinger operator, where $\Delta$ is the Laplacian operator on $\mathbb{R}^{d}$ $(d\geq 3)$, while the nonnegative potential $\mathit{V}(x)$ belongs to the reverse H\"{o}lder class $B_{q}, q>d/2$. In this paper, we study weighted compactness of commutators of some Schr\"{o}dinger operators, which include Riesz transforms, standard Calder\'{o}n-Zygmund operatos and Littlewood-Paley functions. These results generalize substantially some well-know results.
\end{abstract}

\maketitle
\numberwithin{equation}{section}
\newtheorem{theorem}{Theorem}[section]
\newtheorem{lemma}[theorem]{Lemma}
\newtheorem{definition}[theorem]{Definition}
\newtheorem{corollary}[theorem]{Corollary}
\newtheorem{proposition}[theorem]{Proposition}
\newtheorem{claim}[theorem]{Claim}
\newtheorem{remark}[theorem]{Remark}
\allowdisplaybreaks

\section{Introduction}

Let us consider the Schr\"{o}dinger differential operators $\mathcal{L}=-\Delta+\mathit{V}(x) $ on ${\mathbb{R}}^{d}$ with $d\geq3$, where $\Delta$ is the Laplace operator on $\mathbb{R}^{d}$, and the potential $\mathit{V}\geq0$ is a function satisfying the reverse H\"{o}lder inequality(called the reverse H\"{o}lder class $B_{q}$) for some $q>d/2$
\begin{equation}\label{reverse_holder}
\left(\frac{1}{|B|}\int_{B}\mathit{V}(y)^{q}\right)^{1/q}\leq \frac{C}{|B|}\int_{B}\mathit{V}(y)dy
\end{equation}
for any ball $B\subset {\mathbb{R}}^{d}$. The general theory of semigroup, in particlar Yosida's generating theorm \cite{yosida}, implies that $\mathcal{L}$ is the infinitesimal generator of semigroups, formally denoted by $T_{t}=e^{-t\mathcal{L}}$, that solves the diffusion problem
$$\left\{
\begin{aligned}
&\frac{\partial}{\partial t}u(x,t)={\mathcal{L}} u(x,t),\ &(x,t)\in \mathbb R^{d}\times [0,\infty);\\
&u(x,0)=f(x),\ & x\in\mathbb R^{d},
\end{aligned}
\right.$$
by  setting $u(x,t)=e^{-t\mathcal{L}}f(x)$.

In this paper, we study weighted compactness of commutators of some classical operators of harmonic analysis associated with Schr\"{o}dinger operators. In functional analysis, an important branch is the theory of compact operators. Let $L$ be a linear operator from a  Banach space $X$ to another Banach space $Y$. We call $L$ a compact operator if the image under $L$ of any bounded subset of $X$ is a relatively compact subset of $Y$. One of classical examples of compact operators is the compact imbedding of Sobolev spaces. By such imbedding, it can be converted an elliptic boundary value problem into a
Fredholm integral equation. For further information on compact operators, we refer the reader to Conway \cite{conway}, Folland-Stein \cite{folland} and Kutateladze \cite{kutateladze}.

In 1978, Uchiyama \cite{uchiyama} first studied the compactness of commutators of a singular integral operator with the kernel $\Omega\in\textrm{Lip}_{1}({\mathbb{S}}^{d-1})$ defined by
$$T_{\Omega}f(x)=\textrm{p.v.}\int_{\mathbb{R}^{d}}\frac{\Omega(y/|y|)}{|y|^{d}}f(x-y)dy.$$
He obtained that the commutator $[b,T_{\Omega}]$ is compact on $L^{p}(\mathbb{R}^{d}), 1<p<\infty$, if and only if $b\in \textrm{CMO}(\mathbb{R}^{d})$, where $\textrm{CMO}(\mathbb{R}^{d})$ denotes the closure of $\mathcal{C}_{c}^{\infty}(\mathbb{R}^{d})$ in the topology of $\textrm{BMO}(\mathbb{R}^{d})$. In 1984, Janson and Peetre \cite{janson} established the theory of paracommutators $T_{b}$ defined by
$$\widehat{(T_{b}^{s,t}f)}(\xi)=\frac{1}{2\pi}\int_{\mathbb{R}^{d}}\hat{b}(\xi-\eta)A(\xi,\eta)|\xi|^{s}|\eta|^{t}\hat{f}(\eta)d\eta.$$
Under some assumptions of $A(\cdot,\cdot)$, Janson and Peetre proved that if $b\in \textrm{CMO}(\mathbb{R}^{d})$, then $T_{b}^{0,0}$ is compact, see \cite[Theorems 13.2 and 13.3]{janson}. The commutators and higher commutators of convolution singular integrals are special cases of $T_{b}^{0,0}$. Then the result of \cite{janson} is a generalization of that in \cite{uchiyama}.

Since then, the study on the compactness of commutators of different operators has attracted
much more attention. For examples, the compactness of commutators of linear Fourier multipliers and pseudodifferential operators was considered by Cordes \cite{corder}. Peng \cite{peng} gave the the compactness of paracommutators $T_{b}$. Beatrous and Li \cite{beatrous} studied the boundedness and compactness of the commutators of Hankel type operators. Krantz and
Li \cite{krantz1,krantz2} applied the compactness characterization of the commutator $[b,T_{\Omega}]$ to study Hankel type operators on Bergman spaces. Wang \cite{wang1} showed that the commutators of fractional integral operators are compact from $L^{p}(\mathbb{R}^{d})$ to $L^{q}(\mathbb{R}^{d})$. In 2009, Chen and Ding \cite{chen1} proved that the commutator
of singular integrals with variable kernels is compact on $\mathbb{R}^{d}$ if and only if $b \in \textrm{CMO}(\mathbb{R}^{d})$. In \cite{chen1}, the authors also established the compactness of Littlewood-Paley square functions in \cite{chen2}. After that, Chen, Ding and Wang \cite{chen3} obtained the compactness of commutators for Marcinkiewicz integrals on Morrey spaces. Liu and Tang \cite{liu1} studied the compactness for higher order commutators of oscillatory singular integral operators. Li and Peng \cite{li-peng} investigated  compact commutators of Riesz transforms associated to Schr\"{o}dinger operators. Li, Mo and Zhang \cite{li2} established a compactness criterion with applications to the commutators associated with Schr\"{o}dinger operators. For more information about the compactness problems of commutators, see also \cite{benyi1,benyi2,benyi3,bu1,bu2,cao1,cao2,chaffee1,chaffee2,ding,duong,guo,hytonen1,hytonen2,hytonen3,tao,torres1,wang2,wu,xue1,xue2} and the references therein.

The study of Schr\"{o}dinger operator $\mathcal{L}=-\Delta+\mathit{V}$ recently attracted much attention, see \cite{bongioanni1,bongioanni2,bongioanni3,bongioanni4,dziubanski1,dziubanski2,guo1,shen,zhong}. In particular, Shen \cite{shen} considered  $L^{p}$ estimates for Schr\"{o}dinger operators $\mathcal{L}$ with certain potentials which include Schr\"{o}dinger Riesz transforms
$$R_{j}^{\mathcal{L}}=\frac{\partial}{\partial x_{j}}\mathcal{L}^{-{1}/{2}},\qquad j=1,\ldots,n.$$
 Shen also proved that the Schr\"{o}dinger type operators: $\nabla(-\Delta+\mathit{V})^{-1}\nabla$, $\nabla(-\Delta+\mathit{V})^{-1/2}$, $(-\Delta+\mathit{V})^{-1/2}\nabla$ with $\mathit{V}\in B_{d}$, and $(-\Delta+\mathit{V})^{i\gamma}$ with $\gamma\in \mathbb{R}$ and $\mathit{V}\in B_{d/2}$, are standard Calder\'{o}n-Zygmund operators.

Recently, Bongioanni, Harboure and Salinas \cite{bongioanni1} proved $L^{p}(\mathbb{R}^{d})$ $(1<p<\infty)$ boundedness for commutators of Riesz transforms associated with Schr\"odinger
operators with $\textrm{BMO}(\rho)$ functions which include the class of BMO functions. In  \cite{bongioanni2}, Bongioanni et al. established the weighted boundedness for Riesz transforms, fractional integrals and Littlewood-Paley functions associated to Schr\"{o}dinger operators with weight $A_{p}^{\rho}$ class which includes the Muckenhoupt weight class. Tang and his  collaborators \cite{tang1,tang2,tang3} have established weighted norm inequalities for some Schr\"{o}dinger type operators, which include commutators of Riesz transforms, fractional integrals, and Littlewood-Paley functions
related to Schr\"{o}dinger operators (see also \cite{bongioanni3,bongioanni4}).

Naturally,
it will be a very interesting problem to ask whether we can establish the weighted compactness of commutators of some Schr\"{o}dinger type operators with $\textrm{CMO}(\rho)$ functions and weight $A_{p}^{\rho}$ class. In this paper, we give a positive answer. To obtain the conclusion, we will utilize a new estimate of kernels and new weighted inequalities for new maximal
operators. It is worth pointing out that our method is applicable to  more general Schr\"{o}dinger type operators, and generalizes the results obtained in \cite{li-peng,li2}.  The paper is organized as follows. In Section \ref{sec2}, we give some notation and several basic results which will play a crucial role in the sequel. In Section \ref{sec3}, we establish the weighted compactness of commutators of Riesz transforms, standard Calder\'{o}n-Zygmund operators and Littlewood-Paley functions associated with Schr\"odinger operators.

 Throughout this paper, we let $C$ denote constants that are independent of the main parameters involved but whose value may differ from line to line. By $A\sim B$, we mean that there exists a constant $C>1$ such that $1/C\leq A/B\leq C$. By $A\lesssim B$, we mean that there exists a constant $C>0$ such that $ A\leq C B$.

\section{Some notation and basic results}\label{sec2}

We first recall some notation. Given $B=B(x,r)$ and $\lambda>0$, we will write $\lambda B$ for the $\lambda$-dilate ball, which is the ball centered at $x$ and with radius $\lambda r$. Given a Lebesgue measurable set $E$ and a weight $w$,
 the symbol $|E|$  denotes the Lebesgue measure of $E$, and
$$w(E):=\int_{E}wdx.$$
For $0<p<\infty$,
$$\|f\|_{L^{p}(w)}=\left(\int_{\mathbb{R}^{d}}|f(y)|^{p}w(y)dy\right)^{1/p}.$$

The following auxiliary
function $m_{\mathit{V}}(x)$ was first introduced by Shen \cite{shen} and is widely used in the research of Schr\"odinger operators:
$$\rho(x)=\frac{1}{m_{\mathit{V}}(x)}=\sup_{r>0}\left\{r:\,\frac{1}{r^{d-2}}\int_{B(x,r)}V(y)dy\leq 1\right\}.$$

Obviously, $0<m_{\mathit{V}}(x)<\infty$ if $\mathit{V}\neq 0$. In particular, $m_{\mathit{V}}(x)=1$ with $\mathit{V}=1$ and $m_{\mathit{V}}(x)\sim (1+|x|)$ with $\mathit{V}=|x|^{2}$.

For different $x$ and $y$, Shen \cite{shen} gave the following important inequality.
\begin{lemma}\label{m_{v} property_1}
{\rm (\cite[Lemmas 1.4 and 1.8]{shen})}	Assume that $\mathit{V}\in B_{q}$ for $q>d/2$.
\item{\rm (i)} There exist constants $k_{0}>0$, $C_{0}>1$ and $C>0$ such that
	\begin{equation}\label{V_property_1}
	\frac{1}{C_{0}}(1+|x-y|m_{\mathit{V}}(x))^{-k_{0}}\leq\frac{m_{\mathit{V}}(x)}{m_{\mathit{V}}(y)}\leq C_{0}(1+|x-y|m_{\mathit{V}}(x))^{k_{0}/(k_{0}+1)}.
	\end{equation}
	In particular, $m_{\mathit{V}}(x)\sim m_{\mathit{V}}(y)$ if $|x-y|\leq C/m_{\mathit{V}}(x)$.
	\item{\rm (ii)} For $0<r<R<\infty$,
	\begin{equation}\label{V_property_2}
	\frac{1}{r^{d-2}}\int_{B(x,r)}\mathit{V}(y)dy\leq C(R/r)^{d/q-2}\frac{1}{R^{d-2}}\int_{B(x,R)}\mathit{V}(y)dy.
	\end{equation}
\end{lemma}

By $0<m_{\mathit{V}}(x)<\infty$ and $\eqref{V_property_2}$, Guo et al. \cite{guo1} got the following result.
\begin{lemma}\label{estimate_V}
{\rm (\cite[Lemma 1]{guo1})}	Suppose that $\mathit{V}\in B_{q}$ for some $q>d/2$ and let $K>\log_{2}C_{0}+1$, where $C_{0}$ is the constant in $\eqref{V_property_1}$. Then for any $x\in \mathbb{R}^{d}$ and $R>0$, we have
	\begin{equation}\label{V_property_3}
	\frac{1}{(1+m_{\mathit{V}}(x)R)^{K}}\int_{B(x,R)}\mathit{V}(y)dy\leq C R^{d-2}.
	\end{equation}
\end{lemma}

For a number $\theta>0$ and a ball $B=B(x_{0},r)$ with center at $x_{0}$ and radius $r$, we denote $\Psi_{\theta}(B)=(1+r/\rho(x_{0}))^{\theta}$.

A weight will always mean a nonnegative locally integrable function. As in \cite{bongioanni2}, we say that a weight $w$ belongs to the class $A_{p}^{\rho,\theta}, 1<p<\infty$, if there is a constant $C$ such that for all balls $B=B(x,r)$,
$$\left(\frac{1}{\Psi_{\theta}(B)|B|}\int_{B}w(y)dy\right)\left(\frac{1}{\Psi_{\theta}(B)|B|}\int_{B}w^{-{1}/{(p-1)}}(y)dy\right)^{p-1}\leq C.$$
We also say that a nonnegative function $w$ satisfies the $A_{1}^{\rho,\theta}$ condition if there exists a constant $C$ such that for all balls $B$,
$$M_{\mathit{V}}^{\theta}(w)(x)\leq Cw(x)\quad \,\,\text{a.e.}\,\,x\in\mathbb{R}^{d},$$
where
$$M_{\mathit{V}}^{\theta}f(x)=\sup_{B\ni x}\frac{1}{\Psi_{\theta}|B|}\int_{B}|f(y)|dy.$$
Since $\Psi_{\theta}(B)\geq 1$, obviously, $A_{p}\subset A_{p}^{\rho,\theta}$ for $1<p<\infty$, where $A_{p}$ denote the classical Muckenhoupt weights (\cite{garcia-cuerva,muckenhoupt}).

Since
$$\Psi_{\theta}(B)\leq\Psi_{\theta}(2B)\leq 2^{\theta}\Psi_{\theta}(B),$$
we remark that balls can be replaced by cubes in the definitions of $A_{p}^{\rho,\theta}$ for $p\geq 1$ and $M_{\mathit{V}}^{\theta}$,
For convenience, in the rest of this paper, for fixed $\theta>0$,  we use the notation $\Psi(B)$ and $A_{p}^{\rho}$ instead of $\Psi_{\theta}(B)$ and $A_{p}^{\rho,\theta}$, respectively.

The next lemma follows from the definition of $A_{p}^{\rho}$ $(1\leq p<\infty)$:
\begin{lemma}\label{A_p property}
{\rm (\cite{tang1})}	Let $1<p<\infty$. Then the following assertions hold.
		\item{\rm(i)} If $1<p_{1}<p_{2}<\infty$, then $A_{p_{1}}^{\rho}\subset A_{p_{2}}^{\rho}$.
		\item{\rm(ii)} $w\in A_{p}^{\rho}$ if and only if $w^{-{1}/{(p-1)}}\in A_{p^{\prime}}^{\rho}$, where $1/p+1/p^{\prime}=1$.
		\item{\rm (iii)} If $w\in A_{p}^{\rho}$ for $1\leq p<\infty$, then
		$$\frac{1}{\Psi(Q)|Q|}\int_{Q}|f(y)|dy\leq C\left(\frac{1}{w(5Q)}\int_{Q}|f|^{p}w(y)dy\right)^{1/p},$$
		where
		$w(E)=\int_{E}w(x)dx.$
		In particular, letting $f=\chi_{E}$ for any measurable set $E\subset Q$, we have
		\begin{equation}\label{A_p property1}
		\frac{|E|}{\Psi(Q)|Q|}\leq C\left(\frac{w(E)}{w(5Q)}\right)^{1/p}.
		\end{equation}

\end{lemma}

In \cite{bongioanni1}, Bongioanni et al. introduced a new space $\textrm{BMO}(\rho)$ defined by
$$
\|f\|_{\textrm{BMO}(\rho)}=\sup_{B\subset\mathbb{R}^{d}}\frac{1}{\Psi(B)|B|}\int_{B}|f(x)-f_{B}|dx<\infty,
$$
where $f_{B}=\frac{1}{|B|}\int_{B}f(y)dy$, $\Psi(B)=(1+r/\rho(x_{0}))^{\theta}$, $B=B(x_{0},r)$ and $\theta>0$.
We denote by $\textrm{CMO}(\rho)$  the closure of $\mathcal{C}_{c}^{\infty}$ in the  topology of $\textrm{BMO}(\rho)$, where $\mathcal{C}_{c}^{\infty}$ is the set of all smooth functions on $\mathbb{R}^{d}$ with compact supports.

To prove the weighted boundedness for the area functions related with Schr\"{o}dinger operators, Tang et al. \cite{tang3} consider the following variant of maximal operator $M_{V,\eta}, 0<\eta<\infty$, defined as
$$M_{V,\eta}f(x):=\sup_{B\ni x}\frac{1}{(\Psi(B))^{\eta}|B|}\int_{B}|f(y)|dy.$$
One of the main results obtained in \cite{tang3} is the weighted $L^{p}$-boundedness of $M_{V,\eta}, 0<\eta<\infty$. Precisely,
\begin{lemma}\label{maximal_bound}
	Let $1<p<\infty$ and $p^{\prime}=p/(p-1)$. If $w\in A_{p}^{\rho}$, then there exists a constant $C>0$ such that
	$$
	\|M_{V,p^{\prime}}\|_{L^{p}(w)}\leq C\|f\|_{L^{p}(w)}.
	$$
\end{lemma}
\begin{remark}
	If $\eta=1$, Lemma $\ref{maximal_bound}$ holds for $1<p_{0}<p<\infty$.
\end{remark}

There exists many operators related with $-\Delta+V$ are standard Calder\'{o}n-Zygmund operators (\cite{shen}), for instance,
$$\begin{cases}
&\nabla(-\Delta+\mathit{V})^{-1}\nabla,\ \mathit{V}\in B_{n},\\
&\nabla(-\Delta+\mathit{V})^{-1/2},\ \mathit{V}\in B_{n},\\
&(-\Delta+\mathit{V})^{-1/2}\nabla,\ \mathit{V}\in B_{n},\\
&(-\Delta+\mathit{V})^{i\gamma},\ \gamma\in\mathbb{R}\ \&\ \mathit{V}\in B_{n/2},
\end{cases}$$
 and $\nabla^{2}(-\Delta+\mathit{V})^{-1}\nabla$ with $\mathit{V}$ being a nonnegative polynomial. In particular, the kernels $K$ of the operators mentioned above all satisfy the following conditions: for some $\delta_{0}>0$ and any $l\in\mathbb{N}_{0}=\mathbb{N}\bigcup\{0\}$, there exists a constant $C_{l}$ such that
\begin{align}\label{kernel_T1}
|K(x,y)|\leq\frac{C_{l}}{(1+|x-y|(m_{\mathit{V}}(x)+m_{\mathit{V}}(y)))^{l}}\frac{1}{|x-y|^{d}}
\end{align}
and
\begin{align}\label{kernel_T2}
&|K(x+h,y)-K(x,y)|+|K(x,y+h)-K(x,y)|\nonumber\\ &\qquad\leq\frac{C_{l}}{(1+|x-y|(m_{\mathit{V}}(x)+m_{\mathit{V}}(y)))^{l}}\frac{|h|^{\delta_{0}}}{|x-y|^{d+\delta_{0}}}
\end{align}
whenever $x,y,h\in\mathbb{R}^{d}$ and $|h|<|x-y|/2$.

Next we give a result of maximal Calder\'{o}n-Zygmund operators associated with Schr\"{o}dinger type.
\begin{lemma}\label{maximal_T}
Let $1<p<\infty$. If $w\in A_{p}^{\rho}$, then there exists a constant $C>0$ such that
$$\|T^{\ast}f\|_{L^{p}(w)}\leq C\|f\|_{L^{p}(w)},$$
where the maximal operator $T^{\ast}$ is defined by
$$T^{\ast}f(x):=\sup_{\epsilon>0}|T_{\epsilon}f(x)|=\sup_{\epsilon>0}\left|\int_{|y-x|>\epsilon}K(x,y)f(y)dy\right|.$$
\end{lemma}

We remark that the maximal operator can be controlled by $M_{\mathit{V},\eta}$ in $L^{p}(w)$, and it is proved in \cite{tang1}. Thus, using Lemma $\ref{maximal_bound}$, it implies that Lemma $\ref{maximal_T}$ holds.

The following result about the weighted $L^{p}$ boundedness of commutator $[b,T]$which can be found in \cite{tang1}.
\begin{lemma}\label{weighted_commutator_T}
	Let $1<p<\infty$, $b\in{{\rm BMO}}(\rho)$ and $w\in A_{p}^{\rho}$. Then there exists a  constant $C_{p}>0$ such that
	$$\|[b,T]\|_{L^{p}(w)}\leq C_{p}\|b\|_{\textrm{\rm BMO}(\rho)}\|f\|_{L^{p}(w)}.$$
\end{lemma}

Next we consider another class $\mathit{V}\in B_{q}$ for $q\geq d/2$ for Riesz transforms associated with Schr\"{o}dinger operators. Let
$$
T_{1}=(-\Delta+\mathit{V})^{-1}\mathit{V},\quad T_{2}=(-\Delta+\mathit{V})^{-1/2}\mathit{V}^{1/2}\quad\text{\rm and}\quad T_{3}=(-\Delta+\mathit{V})^{-1/2}\nabla.
$$
Tang considered the weighted estimates for the operators $T_{i}, i=1,2,3,$ in \cite{tang1}.
\begin{lemma}\label{weighted_Ti}
Suppose that $\mathit{V}\in B_{q}$ and $q\geq d/2$. Then the following three statements hold.
\item{\rm (i)} If $q^{\prime}\leq p<\infty$ and $w\in A_{p/q^{\prime}}^{\rho}$, then
$\|T_{1}f\|_{L^{p}(w)}\leq C\|f\|_{L^{p}(w)}.$
\item{\rm (ii)} If $(2q)^{\prime}\leq p<\infty$ and $w\in A_{p/(2q)^{\prime}}^{\rho}$, then
$\|T_{2}f\|_{L^{p}(w)}\leq C\|f\|_{L^{p}(w)}.$
\item{\rm (iii)} If $p_{0}^{\prime}\leq p<\infty$ and $w\in A_{p/p_{0}^{\prime}}^{\rho}$, where $1/p_{0}=1/q-1/d$ and $d/2\leq q<d$, then
$$\|T_{3}f\|_{L^{p}(w)}\leq C\|f\|_{L^{p}(w)}.$$
\end{lemma}

In \cite{tang1}, using Lemma $\ref{weighted_Ti}$ and a pointwise estimate of the kernels of $T_{i}, i=1,2,3$, the author got the weighted $L^{p}$ boundedness of commutator $[b,T_{i}], i=1,2,3$, with $b\in{\rm BMO}(\rho)$.

\begin{lemma}\label{weighted_Ti_commutators}
Suppose that $\mathit{V}\in B_{q}, q\geq d/2$. Let $b\in {\rm BMO}(\rho)$. Then the following three statements  hold.
\item{\rm (i)} If $q^{\prime}\leq p<\infty$ and $w\in A_{p/q^{\prime}}^{\rho}$, then
$$\|[b,T_{1}]f\|_{L^{p}(w)}\leq C\|b\|_{{\rm BMO}(\rho)}\|f\|_{L^{p}(w)}.$$
\item{\rm (ii)} If $(2q)^{\prime}\leq p<\infty$ and $w\in A_{p/(2q)^{\prime}}^{\rho}$, then
$$\|[b,T_{2}]f\|_{L^{p}(w)}\leq C\|b\|_{{\rm BMO}(\rho)}\|f\|_{L^{p}(w)}.$$
\item{\rm (iii)} If $p_{0}^{\prime}\leq p<\infty$ and $w\in A_{p/p_{0}^{\prime}}^{\rho}$, where $1/p_{0}=1/q-1/d$ and $d/2\leq q<d$, then
$$\|[b,T_{3}]f\|_{L^{p}(w)}\leq C\|b\|_{{\rm BMO}(\rho)}\|f\|_{L^{p}(w)}.$$
\end{lemma}

We list some estimates of the kernel $K_{i}$ of operator $T_{i}, i=1,2,3$, and refer the reader to Guo-Li-Peng \cite{guo1} and Shen \cite{shen}.
\begin{lemma}\label{T_i kernel}
Suppose $\mathit{V}\in B_{q}$ for some $q>n/2$. Then there exist constants $\delta>0$ and $C_{l}$ such that for $0<|h|<|x-y|/16$  and $l>0$,
\begin{align}\label{K_{1}_property1}
|K_{1}(x,y)|\leq \frac{C_{l}}{(1+|x-y|m_{\mathit{V}}(x))^{l}}\frac{1}{|x-y|^{d-2}}\mathit{V}(y),
\end{align}
\begin{align}\label{K_{1}_property2}
|K_{1}(x+h,y)-K_{1}(x,y)|\leq \frac{C_{l}}{(1+|x-y|m_{\mathit{V}}(x))^{l}}\frac{|h|^{\delta}}{|x-y|^{d-2+\delta}}\mathit{V}(y),
\end{align}
\begin{align}\label{K_{2}_property1}
|K_{2}(x,y)|\leq \frac{C_{l}}{(1+|x-y|m_{\mathit{V}}(x))^{l}}\frac{1}{|x-y|^{d-1}}\mathit{V}^{1/2}(y)
\end{align}
and
\begin{align}\label{K_{2}_property2}
|K_{2}(x+h,y)-K_{2}(x,y)|\leq \frac{C_{l}}{(1+|x-y|m_{\mathit{V}}(x))^{l}}\frac{|h|^{\delta}}{|x-y|^{d-1+\delta}}\mathit{V}^{1/2}(y).
\end{align}
In particular, for $d/2<q<d$, we also have
\begin{align}\label{K_{3}_property1}
&|K_{3}(x,y)|\\
&\leq \frac{C_{l}}{(1+|x-y|m_{\mathit{V}}(x))^{l}}\frac{1}{|x-y|^{d-1}}\left(\int_{B(y,|x-y|)}\frac{\mathit{V}(\xi)}{|y-\xi|^{d-1}}d\xi+\frac{1}{|x-y|}\right)\nonumber
\end{align}
and
\begin{align}\label{K_{3}_property2}
&|K_{3}(x+h,y)-K_{3}(x,y)|\\
&\leq \frac{C_{l}}{(1+|x-y|m_{\mathit{V}}(x))^{l}}\frac{|h|^{\delta}}{|x-y|^{d-1+\delta}}\left(\int_{B(y,|x-y|)}\frac{\mathit{V}(\xi)}{|y-\xi|^{d-1}}d\xi+\frac{1}{|x-y|}\right).\nonumber
\end{align}
\end{lemma}

The maximal operator $T_{i, {\rm Max}}$ of $T_{i}, i=1,2,3,$  is defined by
$$T_{i,{\rm Max}}f(x):=\sup_{r>0}\left|\int_{|x-y|>r}K_{i}(x,y)f(y)dy\right|, \quad i=1,2,3.$$

To prove our results, we need the following weighted boundedness of maximal operator $T_{i,{\rm Max}}$.
\begin{theorem}\label{weighted_maximal_Ti}
Suppose that $\mathit{V}\in B_{q}, q> d/2$. Then the following three statements are hold.
		\item{\rm (i)} If $q^{\prime}< p<\infty$ and $w\in A_{p/q^{\prime}}^{\rho}$, then
		$$\|T_{1,{\rm Max}}f\|_{L^{p}(w)}\leq C\|f\|_{L^{p}(w)}.$$
		\item{\rm (ii)} If $(2q)^{\prime}< p<\infty$ and $w\in A_{p/(2q)^{\prime}}^{\rho}$, then
		$$\|T_{2,{\rm Max}}f\|_{L^{p}(w)}\leq C\|f\|_{L^{p}(w)}.$$
		\item{\rm (iii)} If $p_{0}^{\prime}< p<\infty$ and $w\in A_{p/p_{0}^{\prime}}^{\rho}$, where $1/p_{0}=1/q-1/d$ and $d/2< q<d$, then
		$$\|T_{3,{\rm Max}}f\|_{L^{p}(w)}\leq C\|f\|_{L^{p}(w)}.$$
\end{theorem}

To prove Theorem $\ref{weighted_maximal_Ti}$, we need the following two lemmas. The first is a covering lemma.
\begin{lemma}\label{cover_ball}
{\rm (\cite{tang1})}	For any ball $B=B(x_{0},r)$, if $r\geq 1/m_{\mathit{V}}(x_{0})$, the ball $B$ can be decomposed into finite disjoint cubes $\{Q_{i}\}_{i=1,\ldots,m}$ such that
	$$B\subset\bigcup_{i=1}^{m}Q_{i}\subset 2\sqrt{d}B$$
	and
	$${r_{i}}/{2}\leq\frac{1}{m_{\mathit{V}}(x)}\leq2\sqrt{d}C_{0}r_{i}$$
	for some $x\in Q_{i}=Q(x_{i},r_{i})$, where $C_{0}$ is the same as in $\eqref{V_property_1}$ of Lemma $\ref{m_{v} property_1}$
\end{lemma}

\begin{lemma}\label{ball_mean}
	Suppose that $0<\eta<\infty$ and $\mathit{V}\in B_{q}, q\geq d/2$. For any ball $B=B(x_{0},r)$, we have for $x\in B$
	$$
	\frac{1}{|B|}\int_{B}|f(y)|dy\leq (2\sqrt{d})^{d}M_{\mathit{V},\eta}f(x).
	$$
\end{lemma}
\begin{proof} It is sufficient to consider two cases.

{\it Case 1:} $r<1/m_{\mathit{V}}(x_{0})$.  Since $r<1/m_{\mathit{V}}(x_{0})$ implies $\Psi(B)\sim 1$, this case is easy to handle and we omit the details.

{\it Case 2:} $r\geq 1/m_{\mathit{V}}(x_{0})$. Using Lemma $\ref{ball_mean}$, there exist finite disjoint cubes $Q_{i}(x_{i},r_{i}), i=1,\ldots,m,$ such that
$$\begin{cases}
\int_{B}|f(y)|dy\leq\sum\limits_{i=1}^{m}\int_{Q_{i}}|f(y)|dy,\\
|B|\leq\sum\limits_{i=1}^{m}|Q_{i}|\leq(2\sqrt{d})^{d}|B|,\\
{r_{i}}/{2}\leq{1}/{m_{\mathit{V}}(x_{i})}\leq2\sqrt{d}C_{0}r_{i}.
\end{cases}$$
Note that $r_{i}<C/m_{\mathit{V}}(x_{i})$ implies $\Psi(Q_{i})\sim1$. For $x\in B$, we then have
$$
\int_{B}|f(y)|dy\leq\sum_{i=1}^{m}|Q_{i}|M_{\mathit{V},\eta}f(x)\leq (2\sqrt{d})^{d}|B|M_{\mathit{V},\eta}f(x).$$
This finished the proof.
\end{proof}

\begin{proof}[Proof of Theorem $\ref{weighted_maximal_Ti}$]
We first prove ${\rm(i)}$. Take
$$T_{1,r}f(x)=\int_{|x-y|>r}K_{1}(x,y)f(y)dy.$$
For $B=B(x,r/16)$, we   divide $f$ as  $f=f_{1}+f_{2}$, where $f_{1}:=f\chi_{16B}$. It follows from Lemma $\ref{ball_mean}$ and Lemma $\ref{weighted_Ti}$ (i) with $w=1$ that
\begin{align*}
&|T_{1,r}f(x)|=\frac{1}{|B|}\int_{B}|T_{1,r}f(x)|dy\\
&\leq \frac{1}{|B|}\int_{B}|T_{1}f(y)|dy+ \frac{1}{|B|}\int_{B}|T_{1}f_{1}(y)|dy+\frac{1}{|B|}\int_{B}|T_{1}f_{2}(y)-T_{1,r}f(x)|dy\\
&\leq M_{\mathit{V},\eta}(T_{1}f)(x)+\frac{1}{|B|^{1/q^{\prime}}}\|T_{1}f_{1}\|_{L^{q^{\prime}}}+\frac{1}{|B|}\int_{B}|T_{1}f_{2}(y)-T_{1,r}f(x)|dy\\
&\leq M_{\mathit{V},\eta}(T_{1}f)(x)+C\left(\frac{1}{|B|}\int_{B}|f(y)|^{q^{\prime}}dy\right)^{1/q^{\prime}}+\frac{1}{|B|}\int_{B}|T_{1}f_{2}(y)-T_{1,r}f(x)|dy\\
&\leq M_{\mathit{V},\eta}(T_{1}f)(x)+C\left(M_{\mathit{V},\eta}(|f|^{q^{\prime}})\right)^{1/q^{\prime}}+\frac{1}{|B|}\int_{B}|T_{1}f_{2}(y)-T_{1,r}f(x)|dy.
\end{align*}
For the third term in the last inequaliy, we have
\begin{align*}
&\frac{1}{|B|}\int_{B}|T_{1}f_{2}(y)-T_{1,r}f(x)|dy\\
&=\frac{1}{|B|}\int_{B}\left|\int_{(16)^{c}}K_{1}(y,\xi)f(\xi)d\xi-\int_{|x-\xi|>r}K_{1}(x,\xi)f(\xi)d\xi\right|dy\\
&\leq\frac{1}{|B|}\int_{B}\left(\int_{|x-\xi|>r}|K_{1}(y,\xi)-K_{1}(x,\xi)||f(\xi)|d\xi\right)dy\\
&=:\frac{1}{|B|}\int_{B}I_{1}(y)dy.
\end{align*}
Now, set $h=|y-x|$. Since $|y-x|<r/16<|x-\xi|/16$ for $y\in B$,  by $\eqref{K_{1}_property2}$ in Lemma $\ref{T_i kernel}$, we can obtain that for $l=\theta\eta/q^{\prime}+K$,
\begin{align*}
&I_{1}(y)\leq \sum_{k=0}^{\infty}\int_{2^{k}r<|x-\xi|\leq2^{k+1}r}\frac{C_{l}}{(1+|x-\xi|m_{\mathit{V}}(x))^{l}}\frac{|y-x|^{\delta}}{|x-\xi|^{d-2+\delta}}\mathit{V}(\xi)|f(\xi)|d\xi\\
&\leq \sum_{k=0}^{\infty}\frac{C_{l}\frac{r^{\delta}}{(2^{k}r)^{d-2+\delta}}}{(1+m_{\mathit{V}}(x)2^{k}r)^{l}}\left(\int_{|x-\xi|\leq 2^{k+1}r}\mathit{V}^{q}(\xi)d\xi\right)^{1/q}\left(\int_{|x-\xi|\leq 2^{k+1}r}|f(\xi)|^{q^{\prime}}d\xi\right)^{1/q^{\prime}}\\
&\leq C\sum_{k=0}^{\infty}\frac{C_{l}(M_{\mathit{V},\eta}(|f|^{q^{\prime}})(x))^{1/q^{\prime}}}{(1+m_{\mathit{V}}(x)2^{k}r)^{K}}\frac{r^{\delta}}{(2^{k}r)^{d-2+\delta}}\left(\int_{B(x,2^{k})}V(\xi)d\xi\right)\\
&\leq C(M_{\mathit{V},\eta}(|f|^{q^{\prime}})(x))^{1/q^{\prime}}\sum_{k=0}^{\infty}\frac{r^{\delta}}{(2^{k}r)^{d-2+\delta}}(2^{k}r)^{d-2}\\
&\leq C(M_{\mathit{V},\eta}(|f|^{q^{\prime}})(x))^{1/q^{\prime}}.
\end{align*}
Here we have used $\eqref{V_property_2}$ for $R=2^{k+1}r$, and $\eqref{V_property_3}$ in Lemma $\ref{estimate_V}$ for $R=2^{k}r$. The estimate for $I_{1}(y)$ implies that
$$\frac{1}{|B|}\int_{B}I_{1}(y)dy\leq C(M_{\mathit{V},\eta}(|f|^{q^{\prime}})(x))^{1/q^{\prime}} $$
and
$$T_{1,{\rm Max}}f(x)\leq M_{\mathit{V},\eta}(T_{1}f)(x)+C\left(M_{\mathit{V},\eta}(|f|^{q^{\prime}})\right)^{1/q^{\prime}}.$$
Hence, using  Lemma $\ref{A_p property}$ $(\rm i)$, Lemma $\ref{maximal_bound}$ and  Lemma $\ref{weighted_Ti}$ $(\rm i)$, we have
$$\|T_{1,{\rm Max}}\|_{L^{p}(w)}\leq \|M_{\mathit{V},\eta}(T_{1}f)\|_{L^{p}(w)}+ C\|(M_{\mathit{V},\eta}(|f|^{q^{\prime}})(x))^{1/q^{\prime}}\|_{L^{p}(w)}\leq C\|f\|_{L^{p}(w)}
$$
for $p>q^{\prime}>1$. This finishes the proof of $(\rm i)$.

Similar to $(\rm i)$,  $(\rm ii)$ can be  obtained easily. We omit the details.

It remains to handle the  maximal operator $T_{3,{\rm Max}}$. For any $x$, let
$$T_{3,r}=\int_{|x-y|>r}K_{3}(x,y)f(y)dy.$$
For $B=B(x,r/16)$, we split  $f=f_{1}+f_{2}$, where $f_{1}=f\chi_{16B}$. Similarly, we can obtain
$$T_{3,r}f(x)\leq M_{\mathit{V},\eta}(T_{3}f)(x)+C\left(M_{\mathit{V},\eta}(|f|^{q^{\prime}})\right)^{1/q^{\prime}}+\frac{1}{|B|}\int_{B}I_{3}(y)dy,$$
where $I_{3}(y)$ denotes the following integral:
$$I_{3}(y)=\int_{|x-\xi|>r}|K_{3}(y,\xi)-K_{3}(x,\xi)||f(\xi)|d\xi.$$
Since $y\in B$ and $h=|y-x|<r/16<|x-\xi|$, we deduce from $\eqref{K_{3}_property2}$ that $I_{3}(y)\lesssim I_{3,1}(x)+I_{3,2}(x)$, where
$$\left\{\begin{aligned}
I_{3,1}(x)&:= \int_{|x-\xi|>r}\frac{\frac{|r|^{\delta}}{|x-\xi|^{d-1+\delta}}}{(1+|x-\xi|m_{\mathit{V}}(x))^{l}}\left(\int_{B(\xi,|x-\xi|)}
\frac{\mathit{V}(u)}{|\xi-u|^{d-1}}du\right)|f(\xi)|d\xi;\\
I_{3,2}(x)&:=\int_{|x-\xi|>r}\frac{1}{(1+|x-\xi|m_{\mathit{V}}(x))^{l}}\frac{|r|^{\delta}}{|x-\xi|^{d+\delta}}|f(\xi)|d\xi.
\end{aligned}\right.$$

For  $I_{3,2}(x)$, we have
\begin{align*}
I_{3,2}(x)&\lesssim r^{\delta}\sum_{k=0}^{\infty}\int_{2^{k}r<|x-\xi|\leq 2^{k+1}r}\frac{1}{(1+m_{\mathit{V}}(x)2^{k}r)^{\eta\theta/p_{0}^{\prime}}}\frac{1}{(2^{k}r)^{d+\delta}}d\xi\\
&\lesssim r^{\delta}\sum_{k=0}^{\infty}\frac{1}{(2^{k}r)^{\delta}}(M_{\mathit{V},\eta}(|f|^{p_{0}^{\prime}})(x))^{1/p_{0}^{\prime}}\\
&\lesssim (M_{\mathit{V},\eta}(|f|^{p_{0}^{\prime}})(x))^{1/p_{0}^{\prime}}.
\end{align*}
Since $|u-\xi|<|x-\xi|$ yields $|x-u|\leq |x-\xi|+|\xi-u|<2|x-\xi|$, we can apply  H\"{o}lder's inequality and the Hardy-Littlewood-Sobolev inequality with $\frac{1}{p_{0}}=\frac{1}{q}-\frac{1}{d}$ to obtain
\begin{align*}
&I_{3,1}(x)\leq \sum_{k=0}^{\infty}\int_{2^{k}r<|x-\xi|<2^{k+1}r}\frac{C_{l}\frac{|r|^{\delta}}{(2^{k}r)^{d-1+\delta}}}{(1+m_{\mathit{V}}(x)2^{k}r)^{l}}\left(\int_{B(x,2^{k+2}r)}\frac{\mathit{V}(u)}{|\xi-u|^{d-1}}du\right)|f(\xi)|d\xi\\
&\leq \sum_{k=0}^{\infty}\frac{C_{l}\frac{|r|^{\delta}}{(2^{k}r)^{d-1+\delta}}}{(1+m_{\mathit{V}}(x)2^{k}r)^{l}}\left\|\int_{\mathbb{R}^{d}}\frac{\mathit{V}(u)\chi_{B(x,2^{k+2}r)}}{|\cdot-u|^{d-1}}du\right\|_{L^{p_{0}}} \left(\int_{B(x,2^{k+1}r)}|f(\xi)|^{p_{0}^{\prime}}d\xi\right)^{1/p_{0}^{\prime}}\\
&\leq C \sum_{k=0}^{\infty}\frac{(M_{\mathit{V},\eta}(|f|^{p_{0}^{\prime}})(x))^{1/p_{0}^{\prime}}}{(1+m_{\mathit{V}}(x)2^{k}r)^{K}}\frac{|r|^{\delta}}{(2^{k}r)^{d-1+\delta}}\left(\int_{B(x,2^{k+2}r)}\mathit{V}^{q}(\xi)d\xi\right)^{1/q}(2^{k+1}r)^{d/p_{0}^{\prime}}\\
&\leq C \sum_{k=0}^{\infty}\frac{(M_{\mathit{V},\eta}(|f|^{p_{0}^{\prime}})(x))^{1/p_{0}^{\prime}}}{(1+m_{\mathit{V}}(x)2^{k}r)^{K}}\frac{|r|^{\delta}(2^{k+2}r)^{d/q-n}}{(2^{k}r)^{d-1+\delta}}(2^{k+1}r)^{d/p_{0}^{\prime}}\left(\int_{B(x,2^{k+2}r)}\mathit{V}(\xi)d\xi\right)\\
&\leq C(M_{\mathit{V},\eta}(|f|^{p_{0}^{\prime}})(x))^{1/p_{0}^{\prime}}\sum_{k=0}^{\infty}\frac{|r|^{\delta}}{(2^{k}r)^{d-1+\delta}}(2^{k}r)^{d/q-d+(d/p_{0}^{\prime})+d-2}\\
&\leq C(M_{\mathit{V},\eta}(|f|^{p_{0}^{\prime}})(x))^{1/p_{0}^{\prime}}.
\end{align*}
Here we have used the fact that $d/q-d+(d/p_{0}^{\prime})+d-2=d-1$ and $1/p_{0}=1/q-1/d$. Thus, by a similar manner as the case $(\rm i)$, we obtain the desired result. This completes the proof of Lemma $\ref{weighted_maximal_Ti}$.
\end{proof}

Finally, we continue to investigate the Littlewood-Paley functions related to Schr\"{o}dinger operators. We first introduce some notations. For $(x,t)\in\mathbb{R}_{+}^{d+1}= \mathbb{R}^{d}\times (0,\infty)$, let $T_{s}=e^{-s\mathcal{L}}$  and
$$(Q_{t}f)(x)=t^{2}\left(\frac{dT_{s}}{ds}\bigg|_{s=t^{2}}f\right)(x).$$
The Littlewood-Paley $g$-function $g_{Q}$ and the area function $S_{Q}$ related to Schr\"{o}dinger operators (cf. \cite{benyi2,dziubanski1,tang2,tang3}) are defined by
\begin{equation}\label{g_function}
g_{Q}(f)(x):=\left(\int_{0}^{\infty}|Q_{t}(f)(x)|^{2}\frac{dt}{t}\right)^{{1}/{2}}
\end{equation}
and
\begin{equation}\label{area_function}
S_{Q}(f)(x):=\left(\int_{0}^{\infty}\int_{|x-y|<t}|Q_{t}(f)(y)|^{2}\frac{dydt}{t^{d+1}}\right)^{{1}/{2}}.
\end{equation}
In \cite{bongioanni2,tang3}, the authors proved that the wejghted bounddedness of $g_{Q}$ and $S_{Q}$, respectively.
\begin{lemma}\label{little-paley}
Let $1<p<\infty$. If $w\in A_{p}^{\rho}$, then there exists a constant $C$ such that
$$
\|g_{Q}(f)\|_{L^{p}(w)}\leq C\|f\|_{L^{p}(w)}\,\,\,\text{and}\,\,\,\|S_{Q}(f)\|_{L^{p}(w)}\leq C|f\|_{L^{p}(w)}.
$$
\end{lemma}
The commutators of $g_{Q}$ and $S_{Q}$ with $b\in {\rm BMO}(\rho)$ are defined by
$$g_{Q,b}(f)(x)=\left(\int_{0}^{\infty}|Q_{t}((b(x)-b(\cdot))f)(x)|^{2}\frac{dt}{t}\right)^{{1}/{2}}$$
and
$$S_{Q,b}(f)(x)=\left(\int_{0}^{\infty}\int_{|x-y|<t}|Q_{t}((b(x)-b(\cdot))f)(y)|^{2}\frac{dydt}{t^{d+1}}\right)^{{1}/{2}}.$$

The following lemma contains weighted norm inequalities for the commutators $g_{Q, b}$ and $S_{Q,b}$. This results can be found in \cite{tang2,tang3}.
\begin{lemma}\label{little-paley_commutators}
	Let $b\in{\rm BMO}(\rho)$ and $1<p<\infty$. If $w\in A_{p}^{\rho}$, then there exists a constant $C$ such that
	$$
	\|g_{Q, b}(f)\|_{L^{p}(w)}\leq C\|b\|_{{\rm BMO}(\rho)}\|f\|_{L^{p}(w)}\,\,\,\text{and}\,\,\,\|S_{Q,b}(f)\|_{L^{p}(w)}\leq C\|b\|_{{\rm BMO}(\rho)}\|f\|_{L^{p}(w)}.
	$$
\end{lemma}

In \cite{dziubanski1}, the autors introduce some properties for the integral kernnel $Q_{t}(x,y)$ of the operators $Q_{t}$ in $\eqref{g_function}$ and $\eqref{area_function}$.
\begin{lemma}\label{kernal_property}
	There exist positive constants $c$ and $\delta_{0}\leq 1$ such that for any $l\geq0$ there is a constant $C_{l}$ so that the following inequalities hold:
		\item{\rm (i)} $|Q_{t}(x,y)|\leq C_{l}t^{-n}\left(1+\frac{t}{\rho(x)}+\frac{t}{\rho(y)}\right)^{-l}\exp\left(-{c|x-y|^{2}}/{t^{2}}\right)$,
		\item{\rm (ii)} $|Q_{t}(x+h,y)-Q_{t}(x,y)|\leq C_{l}\left({|h|}/{|t|}\right)^{\delta_{0}}t^{-n}\left(1+\frac{t}{\rho(x)}+\frac{t}{\rho(y)}\right)^{-l}\exp\left(-{c|x-y|^{2}}/{t^{2}}\right)$
		for all $|h|\leq t$.
\end{lemma}

We define the space ${\bf B}=L^{2}(\mathbb{R}^{d+1},dydt/t^{d})$ to be the set of all measurable functions $a$: $\mathbb{R}_{+}^{d+1}\rightarrow \mathbb{C}$ endwoed the norm
$$|a|_{{\bf B}}=\left(\int_{\mathbb{R}_{+}^{d+1}}|a(y,t)|^{2}\frac{dydt}{t^{d}}\right)^{{1}/{2}}<\infty.$$

Let $\varphi$ be a nongegative infinitely differentiable function on $\mathbb{R}_{+}$ such that $\varphi(s)=1$ for $0<s<1$ and $\varphi(x)=0$ for $s\geq 2$. Then the function $\varphi_{t}(x,y):=\frac{1}{t}\varphi(\frac{|x-y|}{t})$ satisfies
$$|\varphi_{t}(x,y)-\varphi(x^{\prime,y})|\leq\frac{|x-x^{\prime}|}{t^{2}}\chi_{[0,2]}\left(\frac{\min\{|x-y|,|x^{\prime}-y|\}}{t}\right)$$
for $|x-y|>2|x-x^{\prime}|$.

Denote by   $\widetilde{K}(\cdot,\cdot)$ the  kernel defined as follows:
$$\widetilde{K}(x,z):=\left\{t^{1/2}\varphi_{t}(x,y)Q_{t}(y,z)\right\}_{(y,z)\in\mathbb{R}_{+}^{d+1}}.$$
It is easy to see that
\begin{equation}\label{kernel_inequality_1}
\left(\int_{0}^{\infty}\int_{|x-y|<t}|Q_{t}(y,z)|^{2}\frac{dydt}{t^{d+1}}\right)^{\frac{1}{2}}\leq |\widetilde{K}(x,z)|_{{\bf B}}.
\end{equation}
The following estimate of  $\widetilde{K}(\cdot,\cdot)$ was obtained by Tang et al. \cite{tang3}.
\begin{lemma}\label{estimate_vector}
	Let $\delta_{0}$ as same as in Lemma $\ref{kernal_property}$. Then for any $l$ we have
	$$|\widetilde{K}(x,z)|_{{\bf B}}\leq\frac{C_{l}}{(1+|x-z|(m_{\mathit{V}}(x)+m_{\mathit{V}}(y)))^{l}}\frac{1}{|x-z|^{d}}.$$
\end{lemma}

We end this section by a general weighted version of the Frechet-Kolmogrov theorem, which was proved by Xue, Yabuta and Yan \cite{xue2}.
\begin{lemma}\label{Frechet-Kolmogrov}
	Let $w$ be a weight on $\mathbb{R}^{d}$. Assume that $w^{-1/(p_{0}-1)}$ is also a weight on $\mathbb{R}^{d}$ for some $p_{0}>1$. Let $0<p<\infty$ and $\mathcal{F}$ be a subset in $L^{p}(w)$, then $\mathcal{F}$ is sequentially compact in $L^{p}(w)$ if the following three conditions are satisfied:
		\item{\rm (i)} $\mathcal{F}$ is bounded, i.e., $\displaystyle\sup_{f\in\mathcal{F}}\|f\|_{L^{p}(w)}<\infty$;
		\item{\rm (ii)} $\mathcal{F}$ uniformly vanishes at infinitly, i.e.,
		$$\lim_{N\rightarrow\infty}\sup_{f\in\mathcal{F}}\int_{|x|>N}|f(x)|^{p}w(x)dx=0;$$
		\item{\rm (iii)} $\mathcal{F}$ is uniformly equicontinous, i.e.,
		$$\lim_{|h|\rightarrow0}\sup_{f\in\mathcal{F}}\int_{\mathbb{R}^{d}}|f(x+h)-f(x)|^{p}w(x)dx=0.$$
\end{lemma}
Note that a operator $\mathcal{T}: V\rightarrow Y$ is said to be a compact operator if $\mathcal{T}$ is continuous and maps bounded subsets into sequentially compact subsets.

\section{Compactness of commutators of Schr\"{o}dinger type operators}\label{sec3}

In this section, we will establish the weighted compactness of commutators of Riesz transforms,  standard Calder\'{o}n-Zygmund operators and Littlewood-Paley functions associated with Schrödinger operators.

\quad

\subsection{The weighted compactness of $[b,T]$}

\quad

\quad

 First of all, we consider the weighted compactness of $[b,T]$.
 \begin{theorem}\label{main_1}
 Let $1<p<\infty$, $b\in{{\rm CMO}}(\rho)$ and $w\in A_{p}^{\rho}$. Then $[b,T]$ is a compact operator from $L^{p}(\mathbb{R}^{d})$ to itself.
 \end{theorem}

\begin{proof}
Since
$$|[b_{1},T]f(x)-[b_{2},T]f(x)|\leq|[b_{1}-b_{2},T]f(x)|$$
and $b\in {\rm CMO}(\rho)\subset {\rm BMO}(\rho)$, then by Lemma $\ref{weighted_commutator_T}$, the commutator $[b,T]$ is continuous on $L^{p}(w)$. Hence, for any bounded set ${F}\subset L^{p}(w)$, where $f\in{F}$ with $\|f\|_{L^{p}(w)}\lesssim 1$, it suffices to prove that
$$\mathcal{F}=\{[b,T]f:\,f\in{F}, b\in {\rm CMO}(\rho)\}$$
is a sequentially compact subset. According to a density argument, if $b\in {\rm CMO}(\rho)$, then there exists a sequence of functions $b_{\epsilon}\in \mathcal{C}_{c}^{\infty}(\mathbb{R}^{d})$ such that
$$\|b-b_{\epsilon}\|_{{\rm BMO}(\rho)}<\epsilon.$$
Thus, by Lemma $\ref{weighted_commutator_T}$, we show that
$$\|[b,T]-[b_{\epsilon,T}]\|_{L^{p}(w)\rightarrow L^{p}(w)}\leq\|[b-b_{\epsilon},T]\|_{L^{p}(w)\rightarrow L^{p}(w)}\lesssim \epsilon.$$
Therefore, it is enough to prove that $\mathcal{F}$ is sequentially compact. Without loss of generalization, we will verify $\mathcal{F}$ satisfies the conditions $(\rm i)-(\rm iii)$ of Lemma $\ref{Frechet-Kolmogrov}$ for $b\in\mathcal{C}_{c}^{\infty}(\mathbb{R}^{d})$. The proof is divided into three steps.

{\it Step I: $\mathcal{F}$ satisfies the condition $(\rm i)$}.
First, by Lemma $\ref{weighted_commutator_T}$, we have
$$\sup_{f\in F}\|[b,T]f\|_{L^{p}(w)}\leq C\|b\|_{{\rm BMO}(\rho)}\|f\|_{L^{p}(w)}<\infty,$$
which yields the fact that the set $\mathcal{F}$ is bounded.

{\it Step II: $\mathcal{F}$ satisfies the condition $(\rm ii)$}. We adapt the method using in \cite{li-peng} to verify the condition $(\rm ii)$ of Lemma $\ref{Frechet-Kolmogrov}$. Assume that $b\in\mathcal{C}_{c}^{\infty}(\mathbb{R}^{d})$ and $\text{ supp } b \subset B(0,R)$, where $B(0,R)$ is a ball of radius $R$ and centered at origin in $\mathbb{R}^{d}$. For $\nu>2$, set $B^{c}=\{x\in\mathbb{R}^{d}:\,|x|>\nu R\}$. Then we have
$$\int_{|x|>\nu R}|[b,T]f(x)|^{p}w(x)dx\leq\int_{|x|>\nu R}\left(\int_{|y|<\rho(0)}|K(x,y)||b(y)||f(y)|dy\right)^{p}w(x)dx.$$

It can be deduced from  $\eqref{V_property_1}$ and the scaling technique directly that for any $x, y\in\mathbb{R}^{d}$ and $c\in(0,1]$,
\begin{eqnarray}\label{m_{v} property_2}
\frac{1}{C_{0}(1+|x-y|m_{\mathit{V}}(y))^{k_{0}+1}}&\lesssim&\frac{1}{1+c|x-y|m_{\mathit{V}}(x)}\\
&\lesssim&\frac{C_{0}}{c(1+|x-y|m_{\mathit{V}}(y))^{{1}/{(1+k_{0})}}},\nonumber
\end{eqnarray}
where the constants $k_{0}$ and $C_{0}$ is as same as in $\eqref{V_property_1}$ of Lemma $\ref{m_{v} property_1}$.

Since $|x|>\nu R$ implies $|x-y|>(1-1/\nu)|x|$ with $\nu>2$,  applying $\eqref{kernel_T1}$ and H\"{o}lder's inequality, we have
\begin{align*}
&\int_{|y|<R}|K(x,y)||b(y)||f(y)|dy\\
&\leq\int_{|y|<R}\frac{C_{l}}{(1+|x-y|(m_{\mathit{V}}(x)+m_{\mathit{V}}(y)))^{l}}\frac{1}{|x-y|^{d}}|b(y)||f(y)|dy\\
&\leq\int_{|y|<R}\frac{C_{l}}{(1+(1-1/\nu)|x|m_{\mathit{V}}(x))^{l}}\frac{1}{((1-1/\nu)|x|)^{d}}|)b(y)||f(y)|dy\\
&\leq\frac{C_{l}\|b\|_{L^{\infty}(\mathbb{R}^{d})}}{(1-1/\nu)^{d}|x|^{d}}\frac{1}{(1+(1-1/\nu)|x|m_{\mathit{V}}(x))^{l}}\left(\int_{|y|<R}|f(y)|dy\right)\\
&\leq \frac{C_{l}\|b\|_{L^{\infty}(\mathbb{R}^{d})}}{(1-1/\nu)^{d}|x|^{d}}\frac{\|f\|_{L^{p}(w)}}{(1+(1-1/\nu)|x|m_{\mathit{V}}(x))^{l}}
\left(\int_{|y|<R}w^{-{1}/{(p-1)}}(y)dy\right)^{1-1/p}.
\end{align*}
Thus, by using $(\ref{m_{v} property_2})$, it follows that
\begin{align*}
&\left(\int_{|x|>\nu R}|[b,T]f(x)|^{p}w(x)dx\right)^{1/p}\\
&\leq C\|b\|_{L^{\infty}(\mathbb{R}^{d})}\|f\|_{L^{p}(w)}\sum_{j=0}^{\infty}\frac{1}{(1-1/\nu)^{d}(2^{j}\nu R)^{d}}\frac{1}{(1+(1-1/\nu)(2^{j}\nu R)m_{\mathit{V}}(0))^{{l}/{(k_{0}+1)}}}\\
&\quad\times \left(\int_{2^{j}\nu R<|x|<2^{j+1}\nu R}w(x)dx\right)^{{1}/{p}}\left(\int_{|y|<R}w^{-{1}/{(p-1)}}(y)dy\right)^{1-1/p}\\
&\leq C\|b\|_{L^{\infty}(\mathbb{R}^{d})}\|f\|_{L^{p}(w)}\sum_{j=0}^{\infty}\frac{1}{(1-1/\nu)^{d}(2^{j}\nu R)^{d}}\frac{1}{(1+(1-1/\nu)(2^{j}\nu R)m_{\mathit{V}}(0))^{{l}/{(k_{0}+1)}}}\\
&\quad\quad\times \left(\int_{B(0,2^{j+1}\nu R)}w(x)dx\right)^{{1}/{p}}\left(\int_{B(0, R)}w^{-{1}/{(p-1)}}(x)dx\right)^{1-1/p}.
\end{align*}
Taking $Q=B(0,2^{j+1}\nu R)$ and $E=B(0,R)$, we  use $\eqref{A_p property1}$ in Lemma $\ref{A_p property}$ to obtain
\begin{align*}
w(5Q)&\leq C\left(\frac{\Psi(Q)|Q|}{|E|}\right)^{p}w(E)\\
&\leq Cw(B(0,R))\left(\frac{(1+2^{j+1}\nu R/\rho(0))^{\theta}(2^{j+1}\nu R)^{d}}{R^{d}}\right)^{p}\\
&\leq Cw(B(0,R))(1+2^{j+1}\nu R/\rho(0))^{\theta p}(2^{j+1}\nu)^{dp}.
\end{align*}
Notice that $1/\rho(0)=m_{\mathit{V}}(0)$, we have
\begin{align*}
&\left(\int_{|x|>\nu R}|[b,T]f(x)|^{p}w(x)dx\right)^{1/p}\\
&\leq C\|b\|_{L^{\infty}(\mathbb{R}^{d})}\|f\|_{L^{p}(w)}\sum_{j=0}^{\infty}\frac{(2^{j+1}\nu)^{d}}{(1-1/\nu)^{d}(2^{j}\nu R)^{d}}\frac{(1+2^{j+1}\nu R/\rho(0))^{\theta }}{(1+(1-1/\nu)(2^{j}\nu R)/\rho(0))^{{l}/{(k_{0}+1)}}}\\
&\quad\quad\times \left(\int_{B(0, R)}w(x)dx\right)^{{1}/{p}}\left(\int_{B(0, R)}w^{-{1}/{(p-1)}}(x)dx\right)^{1-1/p}\\
&\leq C\|b\|_{L^{\infty}(\mathbb{R}^{d})}\|f\|_{L^{p}(w)}\|w\|_{A_{p}^{\rho}}^{{1}/{p}}\frac{1}{(1-1/\nu)^{d+{l}/{(k_{0}+1)}}}\sum_{j=0}^{\infty}\frac{(1+2^{j}\nu R/\rho(0))^{2\theta }}{(1+(2^{j}\nu R)/\rho(0))^{{l}/{(k_{0}+1)}}}.
\end{align*}
Next, we will divide the discussion on the  convergence of the above series into two cases.

{\it Case I}: $R>\rho(0)$. Since $R>\rho(0)$ implies $1/(1+(2^{j}\nu R)/\rho)<1/(1+2^{j}\nu)\leq 1/2^{j}\nu$, if $l>2\theta(k_{0}+1)$, it holds
$$
\sum_{j=0}^{\infty}\frac{(1+2^{j}\nu R/\rho(0))^{2\theta}}{(1+(2^{j}\nu R)/\rho(0))^{{l}/{(k_{0}+1)}}}\leq
\sum_{j=0}^{\infty}\frac{1}{(2^{j}\nu)^{{l}/{(k_{0}+1)}-2\theta}}\leq \frac{C}{\nu^{{l}/{(k_{0}+1)}-2\theta}}.
$$

{\it Case II}: $R\leq\rho(0)$. Note that $R$ and $\rho$ are finite, there exists finite integer $N\geq[\log_{2}(\rho(0)/R)]+1$ such that $2^{N}R>\rho(0)$. Hence, this case goes back to {\it Case I} and we get
$$
\sum_{j=0}^{\infty}\frac{(1+2^{j}\nu R/\rho(0))^{2\theta}}{(1+(2^{j}\nu R)/\rho(0))^{{l}/{(k_{0}+1)}}}\leq
\sum_{j=0}^{\infty}\frac{1}{(2^{j}\nu)^{{l}/{(k_{0}+1)}-2\theta}}\leq \frac{C2^{N({l}/{(k_{0}+1)}-2\theta)}}{\nu^{{l}/{(k_{0}+1)}-2\theta}}.
$$
By the above argument, we obtain
\begin{align*}
&\left(\int_{|x|>\nu R}|[b,T]f(x)|^{p}w(x)dx\right)^{1/p}\\
&\leq C\frac{\|b\|_{L^{\infty}(\mathbb{R}^{d})}\|f\|_{L^{p}(w)}\|w\|_{A_{p}^{\rho}}^{{1}/{p}}}{(1-1/\nu)^{d+{l}/{(k_{0}+1)}}\nu^{{l}/{(k_{0}+1)}-2\theta}}
\max\{2^{N({l}/{(k_{0}+1)}-2\theta)},1\},
\end{align*}
which implies that for any $p>1$ and $l>2\theta(k_{0}+1)$,
$$\lim_{\nu\rightarrow\infty}\int_{|x|>\nu R}|[b,T]f(z)|^{p}w(x)dx=0$$
holds whenever $f\in F$.

{\it Step III: $\mathcal{F}$ satisfies condition $(\rm iii)$}. It remains to show that the set $\mathcal{F}$ is uniformly equicontinuous.
It suffices to verify that for any $\epsilon>0$, if $|h|$ is sufficiently small and only depends  on $\epsilon$, then
\begin{equation}\label{need_prove}
\|[b,T]f(h+\cdot)-[b,T]f(\cdot)\|_{L^{p}(w)}\leq C\epsilon
\end{equation}
holds uniformly for $f\in F$.

For any $x\in\mathbb{R}^{d}$, we divide $[b,T]f(x+h)-[b,T]f(x)=\sum^{4}_{i=1}I_{i}(x)$, where
$$\left\{\begin{aligned}
I_{1}(x)&:=\int_{|x-y|>a|h|}K(x,y)(b(x+h)-b(x))f(y)dy;\\
I_{2}(x)&:=\int_{|x-y|>a|h|}(K(x+h,y)-K(x,y))(b(x+h)-b(y))f(y)dy;\\
I_{3}(x)&:=\int_{|x-y|\leq a|h|}K(x,y)(b(x)-b(y))f(y)dy;\\
I_{4}(x)&:=\int_{|x-y|\leq a|h|}K(x+h,y)(b(x+h)-b(y))f(y)dy.
\end{aligned}\right.$$

Clearly, by the definition of $T^{\ast}$ and $b\in \mathcal{C}_{c}^{\infty}(\mathbb{R}^{d})$, we have $$
|I_{1}(x)|\leq|h|\|\nabla b\|_{L^{\infty}(\mathbb{R}^{d})} T^{\ast}f(x),$$
which, together with  Lemma $\ref{maximal_T}$, indicates that
$$
\|I_{1}\|_{L^{p}(w)}\leq |h|\|T^{\ast}f\|_{L^{p}(w)}\leq C|h|\|f\|_{L^{p}(w)}.
$$

For $I_{2}(x)$, take $a>2$. Using  $\eqref{kernel_T2}$ and $\|b\|_{L^{\infty}(\mathbb{R}^{d})}\leq C$, we have
\begin{align*}
&|I_{2}(x)|\lesssim \int_{|x-y|>a|h|}\frac{1}{(1+|x-y|(m_{\mathit{V}}(x)+m_{\mathit{V}}(y)))^{l}}\frac{|h|^{\delta_{0}}}{|x-y|^{d+\delta_{0}}}|f(y)|dy\\
&\lesssim\sum_{k=0}^{\infty}\int_{2^{k}a|h|<|x-y|\leq2^{k+1}a|h|}\frac{1}{(1+|x-y|m_{\mathit{V}}(x))^{l}}\frac{|h|^{\delta_{0}}}{|x-y|^{d+\delta_{0}}}|f(y)|dy\\
&\lesssim\sum_{k=0}^{\infty}\int_{2^{k}a|h|<|x-y|\leq2^{k+1}a|h|}\frac{1}{(1+(2^{k}a|h|)m_{\mathit{V}}(x))^{l}}\frac{|h|^{\delta_{0}}}{(2^{k}a|h|)^{d+\delta_{0}}}|f(y)|dy\\
&\lesssim\sum_{k=0}^{\infty} \frac{1}{(1+(2^{k}a|h|)m_{\mathit{V}}(x))^{l}}\frac{|h|^{\delta_{0}}}{(2^{k}a|h|)^{d+\delta_{0}}} \int_{B(x,2^{k+1}a|h|)}|f(y)|dy\\
&\lesssim M_{\mathit{V},\eta}f(x)\sum_{k=0}^{\infty}\frac{|h|^{\delta_{0}}}{(2^{k}a|h|)^{\delta_{0}}}\\
&\lesssim a^{-\delta_{0}}M_{\mathit{V},\eta}f(x),
\end{align*}
where we have used the constant $l=\theta\eta$. Hence, it follows from Lemma $\ref{maximal_bound}$ that  for $\eta=p^{\prime}$,  $$\|I_{2}\|_{L^{p}(w)}\leq Ca^{-\delta_{0}}\|M_{\mathit{V},\eta}f\|_{L^{p}(w)}\leq Ca^{-\delta_{0}}\|f\|_{L^{p}(w)}.$$

For  $I_{3}(x)$, applying $\eqref{kernel_T1}$ and $b\in \mathcal{C}_{c}^{\infty}(\mathbb{R}^{d})$, we obtain
\begin{align*}
&|I_{3}(x)|\lesssim\|\nabla b\|_{L^{\infty}(\mathbb{R}^{d})}\int_{|x-y|\leq a|h|}\frac{C_{l}}{(1+|x-y|m_{\mathit{V}}(x))^{l}}\frac{1}{|x-y|^{d-1}}|f(y)|dy\\
&\lesssim\sum_{j=-\infty}^{0}\int_{2^{j-1}a|h|<|x-y|\leq 2^{j}a|h|}\frac{1}{(1+(2^{j-1}a|h|)m_{\mathit{V}}(x))^{l}}\frac{1}{(2^{j-1}a|h|)^{d-1}}|f(y)|dy\\
&\lesssim\sum_{j=-\infty}^{0}\frac{1}{(1+(2^{j-1}a|h|)m_{\mathit{V}}(x))^{l}}\frac{1}{(2^{j-1}a|h|)^{d-1}}\int_{B(x,2^{j}a|h|)}|f(y)|dy\\
&\lesssim M_{\mathit{V},\eta}f(x)\sum_{j=-\infty}^{0}2^{j-1}a|h|\\
&\lesssim a|h|M_{\mathit{V},\eta}f(x),
\end{align*}
where we have used the constant $l=\theta\eta$.  Hence, we use Lemma $\ref{maximal_bound}$ to deduce that $$\|I_{3}\|_{L^{p}(w)}\leq Ca|h|\|M_{\mathit{V},\eta}f\|_{L^{p}(w)}\leq Ca|h|\|f\|_{L^{p}(w)}$$ for $\eta=p^{\prime}$.

The estimate of  $I_{4}(x)$ is  similar to  that of $I_{3}(x)$. Since $|x-y|\leq a|h|$, we have $|x+h-y|\leq (a+1)|h|$. Using the kernel property of $T$ in $\eqref{kernel_T1}$ and $b\in \mathcal{C}_{c}^{\infty}(\mathbb{R}^{d})$, we have
\begin{align*}
&|I_{4}(x)|\lesssim\int_{|x+h-y|\leq (a+1)|h|}\frac{C_{l}}{(1+|x+h-y|m_{\mathit{V}}(x+h))^{l}}\frac{1}{|x+h-y|^{d-1}}|f(y)|dy\\
&\lesssim\sum_{j=-\infty}^{0}\frac{\frac{1}{(2^{j-1}(a+1)|h|)^{d-1}}}{(1+(2^{j-1}(a+1)|h|)m_{\mathit{V}}(x+h))^{l}}\int_{B(x+h,2^{j}(a+1)|h|)}|f(y)|dy\\
&\lesssim M_{\mathit{V},\eta}f(x)\sum_{j=-\infty}^{0}2^{j-1}(a+1)|h|\\
&\lesssim (a+1)|h|M_{\mathit{V},\eta}f(x),
\end{align*}
where we have used the constant $l=\theta\eta$. Thus,
$$\|I_{4}\|_{L^{p}(w)}\leq C(a+1)|h|\|M_{\mathit{V},\eta}f\|_{L^{p}(w)}\leq C(a+1)|h|\|f\|_{L^{p}(w)}.$$
Combining with the estimations of $I_{1}(x)$, $I_{2}(x)$, $I_{3}(x)$ and $I_{4}(x)$, it implies that
\begin{align*}
\|[b,T]f(\cdot+h)-[b,T]f(\cdot)\|_{L^{p}(w)}
&\leq \sum_{i=1}^{4}\|I_{i}\|_{L^{p}(w)}\\
&\leq C(|h|+a^{-\delta_{0}}+a|h|+(a+1)|h|)\|f\|_{L^{p}(w)}.
\end{align*}
Consequently, for any $\epsilon>0$, we can choose $a$ large enough such that $$\max\{{1}/{a^{2}},{1}/{(a+1)^{2}}, {1}/{a^{\delta_{0}}}\}<\epsilon,$$
and set $|h|$ being sufficiently small satisfying $|h|<\min\{1/a^{2},1/(a+1)^{2}\}$. Letting $a\rightarrow\infty$, we can see that  $\mathcal{F}$ is uniformly equicontinous (condition (iii)). This completes the proof of Theorem $\ref{main_1}$.
\end{proof}

\quad

\subsection{The weighted compactness of $g_{b}$ and $S_{Q,b}$}

\quad

\quad


\begin{theorem}\label{main_3}
	Let $1<p<\infty$, $b\in{\rm CMO}(\rho)$ and  $w\in A_{p}^{\rho}$. Then $g_{b}$ and $S_{Q,b}$ are compact operators from $L^{p}(w)$ to itself.
\end{theorem}
\begin{proof} We first prove that $S_{Q,b}$ is a compact operator from $L^{p}(w)$ to itself. Since
$$|S_{Q,b_{1}}f(x)-S_{Q,b_{2}}f(x)|\leq |S_{Q,b_{1}-b_{2}}f(x)|,$$
 by the argument in the proof of Theore $\ref{main_1}$ and Lemma $\ref{little-paley_commutators}$, we only need to prove that for $b\in\mathcal{C}_{c}^{\infty}(\mathbb{R}^{d})$, $\mathcal{S}=\{[b,T_{1}]f:\, f\in F, b\in{\rm BMO}(\rho)\}$ satisfies the conditions $(\rm ii)-(\rm iii)$ of Lemma $\ref{Frechet-Kolmogrov}$. We divide the proof into two steps.
	
	{\it Step I: $\mathcal{S}$ satisfies the condition $(\rm ii)$}.
	
	Suppose {\rm supp} $b\subset \{z:\,|z|<R\}$ and  choose $\nu>2$. For $|x|>\nu R>2R$, we have $b(x)=0$. Therefore, by the Minkowski inequality, we have
	\begin{align*} |S_{Q,b}f(x)|&=\left(\int_{0}^{\infty}\int_{|x-y|<t}\left|\int_{\mathbb{R}^{d}}Q_{t}(y,z)(b(x)-b(z))f(z)dz\right|^{2}\frac{dydt}{t^{d+1}}\right)^{\frac{1}{2}}\\
	&\leq\int_{|z|<R}\left(\int_{0}^{\infty}\int_{|x-y|<t}|Q_{t}(y,z)|^{2}\frac{dydt}{t^{d+1}}\right)^{\frac{1}{2}}|b(z)||f(z)|dz.
	\end{align*}
	Using $\eqref{kernel_inequality_1}$, we show that
	$$
	|S_{Q,b}f(x)|\leq \int_{|z|<R}|\widetilde{K}(x,z)|_{\bf B}|b(z)||f(z)|dz.
	$$
Notice that $|\widetilde{K}(x,z)|_{\bf B}$ and $|K(x,y)|$ are both dominated by
$$\frac{C_{l}}{(1+|x-z|(m_{\mathit{V}}(x)+m_{\mathit{V}}(y)))^{l}}\frac{1}{|x-z|^{n}}.$$
 By Lemma $\ref{estimate_vector}$, following the   argument in {\it Step II} of the proof of Theorem $\ref{main_1}$, we obtain
$$\left(\int_{|x|>\nu R}|S_{Q,b}f(x)|^{p}w(x)dx\right)^{\frac{1}{p}}\leq C\frac{\|b\|_{L^{\infty}(\mathbb{R}^{d})}\|f\|_{L^{p}(w)}\|w\|_{A_{p}^{\rho}}^{\frac{1}{p}}}{(1-1/\nu)^{d+\frac{l}{k_{0}+1}}\nu^{\frac{l}{k_{0}+1}-2\theta}}\max\{2^{N(\frac{l}{k_{0}+1}-2\theta)},1\}.$$
For $f\in F$, letting $\nu\rightarrow\infty$ reaches
$$
\int_{|x|>\nu R}|S_{Q,b}f(x)|^{p}w(x)dx\rightarrow0,
$$
 i.e., $\mathcal{S}$ satisfies the condition $(\rm ii)$ in Lemma $\ref{Frechet-Kolmogrov}$.

{\it Step II: $\mathcal{S}$ satisfies the condition $(\rm iii)$}.\quad It suffices to prove that for $1<p<\infty$ and $w\in A_{p}^{\rho}$,
$$\lim_{|h|\rightarrow0}\|S_{Q,b}f(\cdot+h)-S_{Q,b}f(\cdot)\|_{L^{p}(w)}=0.$$
Note that
$$|S_{Q,b}f(x+h)-S_{Q,b}f(x)|\leq \left(\int_{0}^{\infty}\int_{|x-y|<t}|D(x,y,t)|^{2}\frac{dydt}{t^{d+1}}\right)^{\frac{1}{2}},$$
where
$$
D(x,y,t)=\int_{\mathbb{R}^{d}}Q_{t}(y+h,z)(b(x+h)-b(z))f(z)dz-\int_{\mathbb{R}^{d}}Q_{t}(y,z)(b(x)-b(z))f(z)dz.
$$
For any $x\in\mathbb{R}^{d}$,  choose $a>0$, and  write $D(x,y,y)=D_{1}+D_{2}+D_{3}+D_{4}$, where
$$\left\{\begin{aligned}
D_{1}(x)&:=\int_{|x-z|>a|h|}Q_{t}(y,z)(b(x+h)-b(z))f(z)dz;\\
D_{2}(x)&:=\int_{|x-z|>a|h|}(Q_{t}(y+h,z)-Q_{t}(y,z))(b(x+h)-b(z))f(z)dz;\\
D_{3}(x)&:=\int_{|x-z|\leq a|h|}Q_{t}(y,z)(b(x)-b(z))f(z)dz;\\
D_{4}(x)&:=\int_{|x-z|\leq a|h|}Q_{t}(y+h,z)(b(x+h)-b(z))f(z)dz.
\end{aligned}\right.$$
For $D_{3}$ and $D_{4}$, it can be deduced from $\eqref{kernel_inequality_1}$ and Minkowski's inequality that
$$\left(\int_{0}^{\infty}\int_{|x-y|<t}|D_{3}|^{2}\frac{dydt}{t^{d+1}}\right)^{\frac{1}{2}}\leq \int_{|x-z|\leq a|h|}|\widetilde{K}(x,z)|_{\bf B}|b(x)-b(z)||f(z)|dz$$
and
$$\left(\int_{0}^{\infty}\int_{|x-y|<t}|D_{4}|^{2}\frac{dydt}{t^{d+1}}\right)^{\frac{1}{2}}\leq \int_{|x-z|\leq a|h|}|\widetilde{K}(x+h,z)|_{\bf B}|b(x+h)-b(z)||f(z)|dz.$$
Since $|\widetilde{K}(x,z)|_{\bf B}$ and $|K(x,y)|$ are  both dominated by
$$\frac{C_{l}}{(1+|x-z|(m_{\mathit{V}}(x)+m_{\mathit{V}}(y)))^{l}}\frac{1}{|x-z|^{n}},$$
following  the  argument in the proof of Theorem $\ref{main_1}$, we use Lemma $\ref{estimate_vector}$ to obtain \begin{equation}\label{D_3}
\left\|\left(\int_{0}^{\infty}\int_{|x-y|<t}|D_{3}|^{2}\frac{dydt}{t^{d+1}}\right)^{\frac{1}{2}}\right\|_{L^{p}(w)}\leq Ca|h|\|f\|_{L^{p}(w)}
\end{equation}
and
\begin{equation}\label{D_4}
\left\|\left(\int_{0}^{\infty}\int_{|x-y|<t}|D_{4}|^{2}\frac{dydt}{t^{d+1}}\right)^{\frac{1}{2}}\right\|_{L^{p}(w)}\leq C(a+1)|h|\|f\|_{L^{p}(w)}.
\end{equation}
For $D_{2}$, by Minkowski's inequality, we have
\begin{align*}
&\left(\int_{0}^{\infty}\int_{|x-y|<t}|D_{2}|^{2}\frac{dydt}{t^{d+1}}\right)^{\frac{1}{2}}\leq\int_{|x-z|> a|h|}K_{Q}(x,z)|b(x+h)-b(z)||f(z)|dz,
\end{align*}
where
$$K_{Q}(x,z)=\left(\int_{0}^{\infty}\int_{|x-y|<t}|Q_{t}(y+h,z)-Q_{t}(y,z)|^{2}\frac{dydt}{t^{d+1}}\right)^{\frac{1}{2}}.$$
We claim that for $|x-z|>2h$
\begin{equation}\label{S_kernel}
|K_{Q}(x,z)|\leq\frac{C_{l}}{(1+|x-z|(m_{\mathit{V}}(x)+m_{\mathit{V}}(z)))^{l}}\frac{|h|^{\delta_{0}}}{|x-z|^{n+\delta_{0}}},
\end{equation}
where $\delta_{0}$ be as in Lemma $\ref{kernal_property}$. In fact, consider $|x-z|>2h$ and define $E=\{y:\,|y-x|\geq|x-z|/2\}$. Hence, we can apply $(\rm ii)$ of Lemma $\ref{kernal_property}$ to obtain
\begin{align*}
|K_{Q}(x,z)|^{2}&\leq \int_{\mathbb{R}^{d}}\int_{|y-x|/2}|Q_{t}(y+h,z)-Q_{t}(y,z)|^{2}\frac{dtdy}{t^{d+1}}\\
&\leq C|h|^{\delta_{0}}\int_{\mathbb{R}^{d}}\int_{|y-x|/2}\frac{\rho(x)^{2l}}{t^{3d+1+2\delta_{0}+2l}}\frac{t^{2N}}{(t+|z-y|)^{2N}}dtdy\\
&\leq C|h|^{\delta_{0}}\int_{E}\int_{|y-x|/2}\frac{\rho(x)^{2l}}{t^{3d+1+2\delta_{0}+2l}}\frac{t^{2N}}{(t+|z-y|)^{2N}}dtdy\\
&\quad+C|h|^{\delta_{0}}\int_{E^{c}}\int_{|y-x|/2}\frac{\rho(x)^{2l}}{t^{3d+1+2\delta_{0}+2l}}\frac{t^{2N}}{(t+|z-y|)^{2N}}dtdy\\
&=:III_{1}+III_{2}.
\end{align*}
For $III_{1}$, we then have
$$III_{1}\leq C|h|^{\delta_{0}}\int_{E}\frac{\rho(x)^{2l}}{|y-x|^{3d+2\delta_{0}+2l}}dy\leq C\left(\frac{|x-z|}{\rho(x)}\right)^{-2l}\frac{|h|^{2\delta_{0}}}{|x-z|^{2d+2\delta_{0}}}.$$
If $y\in E^{c}$, then $|y-x|<|x-z|/2<|y-z|<2|x-z|$, and hence
\begin{align*}
III_{2}&\leq |h|^{2\delta_{0}}\int_{E^{c}}\left(\int_{|y-x|/2}^{|x-z|}+\int_{|x-z|}^{\infty}\right)\frac{\rho(x)^{2l}}{t^{3d+1+2\delta_{0}+2l}}\frac{t^{2N}}{(t+|z-y|)^{2N}}dtdy\\
&=:III_{2a}+III_{2b}.
\end{align*}
For $III_{2a}$  and $III_{2b}$, letting $d+\delta_{0}<N-l<(3d+2\delta_{0})/2$, we can get
\begin{align*}
III_{2a}&\leq C\frac{|h|^{2\delta_{0}}}{|x-z|^{2N}}\int_{E^{c}}\int_{|y-x|/2}^{\infty}\frac{\rho(x)^{2l}}{t^{3d+1+2\delta_{0}-2N+2l}}dtdy\\
&\leq C\frac{|h|^{2\delta_{0}}}{|x-z|^{2N}}\int_{|y-x|<|x-z|/2}\frac{\rho(x)^{2l}}{t^{3d+2\delta_{0}-2N+2l}}dtdy\\
&\leq C\left(\frac{|x-z|}{\rho(x)}\right)^{-2l}\frac{|h|^{2\delta_{0}}}{|x-z|^{2d+2\delta_{0}}},
\end{align*}
 and
\begin{align*}
III_{2b}&\leq C|h|^{\delta_{0}}\int_{E^{c}}\int_{|x-z|}^{\infty}\frac{\rho(x)^{2l}}{t^{3d+1+2\delta_{0}+2l}}dtdy\\
&\leq C\frac{\rho(x)^{2l}|h|^{2\delta_{0}}}{|x-z|^{3d+2\delta_{0}+2l}}\int_{|y-x|<|x-z|/2}dy\\
&\leq C\left(\frac{|x-z|}{\rho(x)}\right)^{-2l}\frac{|h|^{2\delta_{0}}}{|x-z|^{2d+2\delta_{0}}}.
\end{align*}
Combining the above inequalities, we obtain the desired inequality $\eqref{S_kernel}$.

Taking $a>2$, similar to estimate $I_{2}(x)$ in {\it Step III} of proof of Theorem $\ref{main_1}$, we obtain
\begin{equation}\label{D_2}
\left\|\left(\int_{0}^{\infty}\int_{|x-y|<t}|D_{2}|^{2}\frac{dydt}{t^{d+1}}\right)^{\frac{1}{2}}\right\|_{L^{p}(w)}\leq Ca^{-\delta_{0}}\|f\|_{L^{p}(w)}.
\end{equation}

It remains to estimate $D_{1}$.  Since
\begin{align*}
&\left(\int_{0}^{\infty}\int_{|x-y|<t}|D_{1}|^{2}\frac{dydt}{t^{d+1}}\right)^{\frac{1}{2}}\\
&\leq |b(x+h)-b(x)|\left(\int_{0}^{\infty}\int_{|x-y|<t}\left|\int_{|x-z|>a|h|}Q_{t}(y,z)f(z)dz\right|^{2}\frac{dydt}{t^{d+1}}\right)^{\frac{1}{2}}\\
&=:|b(x+h)-b(x)| S_{Q,a,h}f(x),
\end{align*}
we claim that
\begin{equation}\label{S_Q}
S_{Q,a,h}f(x)\lesssim M_{\mathit{V},\eta}(S_{Q}f)(x)+(M_{\mathit{V},\eta}(|f|^{q_{0}})(x))^{\frac{1}{q_{0}}}+M_{\mathit{V},\eta}f(x),
\end{equation}
where $1<q_{0}<p$. Let $B$ denote the ball centered at $x$ and with radius $r=a|h|/2$. Furthermore, let $f_{1}=f\chi_{2B}$ and $f_{2}=f-f_{1}$. Thus, by Lemma $\ref{ball_mean}$, we get
\begin{align*}
S_{Q,a,h}f(x)&=\frac{1}{|B|}\int_{B}S_{Q,a,h}f(x)d\xi\\
&\leq\frac{1}{|B|}\int_{B}S_{Q}f(\xi)d\xi+ \frac{1}{|B|}\int_{B}S_{Q}f_{1}(\xi)d\xi\\
&\quad+\frac{1}{|B|}\int_{B}|S_{Q}f_{2}(\xi)-S_{Q,a,h}f(x)|d\xi\\
&\leq M_{\mathit{V},\eta}(S_{Q}f)(x)+L_{1}+L_{2}.
\end{align*}
By Lemma $\ref{ball_mean}$, Lemma $\ref{little-paley}$ with $w=1$ and H\"{o}lder's inequality, for any $1<q_{0}<p$, we show that
\begin{align*}
L_{1}&\leq \frac{1}{|B|^{1/q_{0}}}\left(\int_{B}|S_{Q}f_{1}(\xi)|^{q_{0}}d\xi\right)^{\frac{1}{q_{0}}}\\
&\leq C\frac{1}{|B|^{1/q_{0}}}\left(\int_{B}|f_{1}(\xi)|^{q_{0}}d\xi\right)^{\frac{1}{q_{0}}}\leq C(M_{\mathit{V},\eta}(|f|^{q_{0}})(x))^{\frac{1}{q_{0}}}.
\end{align*}
Now, let us estimate $L_{2}$. By the Minkowski inequality, we have
\begin{align*}
&|S_{Q}f_{2}(\xi)-S_{Q,a,h}f(x)|\\
&\leq \left|\left(\int_{0}^{\infty}\int_{|x-y|<t}\left|\int_{|x-z|>a|h|}Q_{t}(y+\xi-x,z)f(z)dz\right|^{2}\frac{dydt}{t^{d+1}}\right)^{\frac{1}{2}}\right.\\
&\quad\quad\left.-\left(\int_{0}^{\infty}\int_{|x-y|<t}\left|\int_{|x-z|>a|h|}Q_{t}(y,z)f(z)dz\right|^{2}\frac{dydt}{t^{d+1}}\right)^{\frac{1}{2}}\right|\\
&\leq\left(\int_{0}^{\infty}\int_{|x-y|<t}\left|\int_{|x-z|>a|h|}|Q_{t}(y+\xi-x,z)-Q_{t}(y,z)||f(z)|dz\right|^{2}\frac{dydt}{t^{d+1}}\right)^{\frac{1}{2}}\\
&\leq \int_{|x-z|>a|h|}\left(\int_{0}^{\infty}\int_{|x-y|<t}|Q_{t}(y+\xi-x,z)-Q_{t}(y,z)|\frac{dydt}{t^{d+1}}\right)^{\frac{1}{2}}|f(z)|dz.
\end{align*}
Since $a>2$ and $|\xi-x|<a|h|/2$ implies $|x-z|>2|\xi-x|$, it holds
\begin{align*}
&\left(\int_{0}^{\infty}\int_{|x-y|<t}|Q_{t}(y+\xi-x,z)-Q_{t}(y,z)|\frac{dydt}{t^{d+1}}\right)^{\frac{1}{2}}\\
&\leq \frac{C_{l}}{(1+|x-z|(m_{\mathit{V}}(x)+m_{\mathit{V}}(z)))^{l}}\frac{|\xi-x|^{\delta_{0}}}{|x-z|^{d+\delta_{0}}},
\end{align*}
which implies that
\begin{align*}
&|S_{Q}f_{2}(\xi)-S_{Q,a,h}f(x)|\\
&\leq\int_{|x-z|>a|h|}\frac{C_{l}}{(1+|x-z|(m_{\mathit{V}}(x)+m_{\mathit{V}}(z)))^{l}}\frac{|\xi-x|^{\delta_{0}}}{|x-z|^{d+\delta_{0}}}|f(z)|dz\\
&\leq C\sum_{j=0}^{\infty}\int_{2^{j}a|h|<|x-z|\leq 2^{j+1}a|h|}\frac{1}{(1+(2^{j-1}a|h|)m_{\mathit{V}}(x))^{l}}\frac{(a|h|)^{\delta_{0}}}{(2^{j-1}a|h|)^{d+\delta_{0}}}|f(z)|dz\\
&\leq C\sum_{j=0}^{\infty}\int_{B(x,2^{j}a|h|)}\frac{1}{(1+(2^{j-1}a|h|)m_{\mathit{V}}(x))^{l}}\frac{(a|h|)^{\delta_{0}}}{(2^{j-1}a|h|)^{d+\delta_{0}}}|f(z)|dz\\
&\leq CM_{\mathit{V},\eta}f(x)\sum_{j=0}^{\infty}\frac{1}{2^{(j-1)\delta_{0}}}\\
&\leq CM_{\mathit{V},\eta}f(x).
\end{align*}
Hence, we obtain $L_{2}\leq CM_{\mathit{V},\eta}f(x)$. Combining $L_{1}$ and $L_{2}$, the claim $\eqref{S_Q}$ holds. Now we are ready to give the estimates of $D_{1}$. Noticing $b\in\mathcal{C}_{c}^{\infty}(\mathbb{R}^{d})$, we have $|b(x+h)-b(x)|\leq C|h|$. Then, applying the claim $\eqref{S_Q}$ for $1<q_{0}<p$, Lemma $\ref{little-paley}$ and Lemma $\ref{maximal_bound}$, we have
\begin{equation}\label{D_1}
\left\|\left(\int_{0}^{\infty}\int_{|x-y|<t}|D_{1}|^{2}\frac{dydt}{t^{d+1}}\right)^{\frac{1}{2}}\right\|_{L^{p}(w)}\leq C|h|\|f\|_{L^{p}(w)}.
\end{equation}
It follows from $\eqref{D_3}$, $\eqref{D_4}$, $\eqref{D_2}$ and $\eqref{D_1}$ that
$$\sup_{f\in F}\|S_{Q,b}f(\cdot+h)-S_{Q,b}f(\cdot)\|_{L^{p}(w)}\leq C(|h|+a|h|+(a+1)|h|+a^{-\delta_{0}})\|f\|_{L^{p}(w)}.$$
Consequently, for any $\epsilon>0$, we can choose $a$ large enough such that $$\max\{{1}/{a^{2}},{1}/{(a+1)^{2}}, {1}/{a^{\delta_{0}}}\}<\epsilon,$$
and set $|h|$ being small enough such that $|h|<\min\{1/a^{2},1/(a+1)^{2}\}$. Letting $a\rightarrow\infty$, we have the uniformly equicontinous (condition (iii)) of $\mathcal{S}$.

Next we will prove that $g_{b}$ is a compact operator from $L^{p}(w)$ to itself. In order to this result, we need the following claim:
\begin{claim}\label{g_kernel}
Let $\delta_{0}$ as same as in Lemma $\ref{kernal_property}$. Then for any $l$ we have
\begin{equation*}
\left(\int_{0}^{\infty}|Q_{t}(x,y)|^{2}\frac{dt}{t}\right)^{\frac{1}{2}}\leq\frac{C_{l}}{(1+|x-y|(m_{\mathit{V}}(x)+m_{\mathit{V}}(y)))^{l}}\frac{1}{|x-y|^{d}}
\end{equation*}
and
\begin{equation*}
\left(\int_{0}^{\infty}|Q_{t}(x,y)-Q_{t}(\xi,y)|^{2}\frac{dt}{t}\right)^{\frac{1}{2}}\leq\frac{C_{l}}{(1+|x-y|(m_{\mathit{V}}(x)+m_{\mathit{V}}(y)))^{l}}\frac{|x-\xi|^{\delta_{0}}}{|x-y|^{d+\delta_{0}}}
\end{equation*}
for $|x-y|>2|x-\xi|$.
\end{claim}
If Claim $\ref{g_kernel}$ holds,  then similar to the proof of $S_{Q,b}$, we can easily obtain  $g_{b}$ is a compact operator on $L^{p}(w)$.  Now we proceed to prove Claim $\ref{g_kernel}$. Using inequality $(\rm a)$ in Lemma $\ref{kernal_property}$, we obtain
\begin{align*}
\int_{0}^{\infty}|Q_{t}(x,y)|^{2}\frac{dt}{t}&\leq \int_{0}^{\infty}\frac{\rho(x)^{2l}}{t^{2d+1+2l}}\frac{t^{2N}}{(t+|x-y|)^{2N}}dt\\
&\leq \left(\int_{0}^{|x-y|}+\int_{|x-y|}^{\infty}\right)\frac{\rho(x)^{2l}}{t^{2d+1+2l}}\frac{t^{2N}}{(t+|x-y|)^{2N}}dt\\
&=:W_{1}+W_{2}.
\end{align*}
For $W_{1}$ and $W_{2}$, let  $d+l<N<(2d+2l+1)/2$. We have the following estimates:
$$W_{1}\leq\frac{\rho(x)^{2l}}{|x-y|^{2N}}\int_{0}^{|x-y|}\frac{1}{t^{2d+1+2l-2N}}dt\leq C\left(\frac{|x-y|}{\rho(x)}\right)^{-2l}\frac{1}{|x-y|^{2d}}$$
and
$$W_{2}\leq \int_{|x-y|}^{\infty} \frac{\rho(x)^{2l}}{t^{2d+1+2l}}\leq C\left(\frac{|x-y|}{\rho(x)}\right)^{-2l}\frac{1}{|x-y|^{2d}}.$$
Combining the above inequalities, we get the first inequality of Claim $\ref{g_kernel}$. The second inequality of Claim $\ref{g_kernel}$ is similarly to be proved, we omit the details.

By the above arguments, we finish  the proof of Theorem $\ref{main_3}$.
\end{proof}

\subsection{The weighted compactness of $[b,T_{i}], i=1,2,3$}

\quad

\quad

Next, we discuss the weighted compactness of $[b,T_{i}], i=1,2,3,$ on $L^{p}(w)$.
\begin{theorem}\label{main_2}
	Suppose that $\mathit{V}\in B_{q}, q> d/2$. Let $b\in {\rm CMO}(\rho)$. Then the following three statements hold.
	\item{\rm (i)} If $q^{\prime}\leq p<\infty$ and $w\in A_{p/q^{\prime}}^{\rho}$, $[b,T_{1}]$ is a compact operator from $L^{p}(w)$ to itself.
	
	\item{\rm (ii)} If $(2q)^{\prime}< p<\infty$ and $w\in A_{p/(2q)^{\prime}}^{\rho}$,
	$[b,T_{2}]$ is a compact operator from $L^{p}(w)$ to itself.
	\item{\rm (iii)} If $p_{0}^{\prime}< p<\infty$ and $w\in A_{p/p_{0}^{\prime}}^{\rho}$, where $1/p_{0}=1/q-1/d$ and $d/2< q<d$,
	$[b,T_{3}]$ is a compact operator from $L^{p}(w)$ to itself.
\end{theorem}

\begin{proof}
	Similar to the proof of Theorem $\ref{main_1}$, and following the process of the proofs of Theorem 2.1, Theorem 2.5 and Theorem 2.7 in \cite{li-peng}, we can easily obtain the desired results. Hence, we omit the details.
\end{proof}

Let
$$T_{1}^{\ast}=\mathit{V}(-\Delta+V)^{-1}, \quad T_{2}^{\ast}=\mathit{V}^{1/2}(-\Delta+\mathit{V})^{1/2}\quad\text{and}\quad T_{3}^{\ast}=\nabla(-\Delta+\mathit{V})^{-1/2}.$$
By duality, the following weighted  $L^{p}$-compactness of $T^{\ast}_{i}$, $i=1,2,3$ can be deduced from Theorem $\ref{main_2}$ immediately.
\begin{corollary}\label{corollary_1}
	Suppose that $\mathit{V}\in B_{q}$ and $q> d/2$. Let $b\in {\rm CMO}(\rho)$. Then the following three statements are hold.
	\item{\it (i)} If $1<p<q$ and $w^{-\frac{1}{p-1}}\in A_{p^{\prime}/q^{\prime}}$, $[b,T_{1}^{\ast}]$ is a compact operator from $L^{p}(w)$ to itself.
	
	\item{\it (ii)} If $1<p<2q$ and $ w^{-\frac{1}{p-1}}\in A_{p^{\prime}/(2q^{\prime})}$,
	$[b,T_{2}^{\ast}]$ is a compact operator from $L^{p}(w)$ to itself.
	\item{\it (iii)} If $1<p<p_{0}$ and $w^{-\frac{1}{p-1}}\in A_{p^{\prime}/p_{0}^{\prime}}$, where $1/p_{0}=1/q-1/d$ and $d/2< q<d$,
	$[b,T_{3}]$ is a compact operator from $L^{p}(w)$ to itself.
\end{corollary}

\quad

\subsection*{Acknowledgements}
This work was in part supported by National Natural Science Foundation of China (Nos. 11871293, 11871452, 12071473, 12071272) and Shandong Natural Science Foundation of China (No. ZR2017JL008)

\end{document}